\documentclass[a4paper]{amsart}
\usepackage{amsmath,amsthm,amssymb,latexsym,epic,bbm,comment,color}
\usepackage{graphicx,enumerate,stmaryrd}
\usepackage[all,2cell]{xy}
\xyoption{2cell}

\newtheorem{theorem}{Theorem}
\newtheorem{lemma}[theorem]{Lemma}
\newtheorem{corollary}[theorem]{Corollary}
\newtheorem{proposition}[theorem]{Proposition}

\usepackage[all]{xy}
\usepackage[active]{srcltx}
\usepackage[parfill]{parskip}

\newcommand{\tto}{\twoheadrightarrow}

\begin{document}
\title[Canonical Gelfand-Zeitlin modules]
{Canonical Gelfand-Zeitlin modules over orthogonal Gelfand-Zeitlin algebras}

\author{Nick Early, Volodymyr Mazorchuk and Elizaveta Vishnyakova}

\begin{abstract}
We prove that every orthogonal Gelfand-Zeitlin algebra $U$ acts 
(faithfully) on its
Gelfand-Zeitlin subalgebra $\Gamma$. Considering the dual module,
we show that every Gelfand-Zeitlin character of $\Gamma$ is realizable in 
a $U$-module. We observe that the Gelfand-Zeitlin formulae can be rewritten 
using divided difference operators. It turns out that the action of the
latter operators on $\Gamma$ gives rise to an explicit basis in a certain
Gelfand-Zeitlin submodule of the dual module mentioned above. This gives,
generically,  both in the case of regular and singular  Gelfand-Zeitlin 
characters, an explicit construction of 
simple modules which realize the given Gelfand-Zeitlin characters.
\end{abstract}

\maketitle

\section{Introduction and description of the results}\label{s1}

In this paper we work over the field $\mathbb{C}$ of complex numbers.
For a positive integer $n\geq 1$, consider the flag 
\begin{displaymath}
\mathfrak{gl}_1\subset \mathfrak{gl}_2\subset\dots \subset 
\mathfrak{gl}_{n-1}\subset \mathfrak{gl}_n
\end{displaymath}
of general linear Lie algebras where each $\mathfrak{gl}_i$ is embedded into $\mathfrak{gl}_{i+1}$ in the obvious
way with respect to the top left corner. This flag induces a flag
\begin{displaymath}
U_1\subset U_2\subset\dots\subset U_{n-1}\subset U_n 
\end{displaymath}
of embeddings of the corresponding universal enveloping algebras, where $U_i$ denotes the 
universal enveloping algebra of $\mathfrak{gl}_i$. Let $Z_i$ be the center of $U_i$. Then the 
subalgebra $\Gamma_n$ of $U_n$, generated by all $Z_i$, where $i=1,2,\dots,n$, is called the 
{\em Gelfand-Zeitlin\footnote{The surname Zeitlin appears in the literature in different spellings,
that is in different latinizations of the Cyrillic version of the original Latin (German) surname, 
in particular, it was spelled as Cetlin, Zetlin, Tzetlin and Tsetlin. Here we use the original Latin spelling.}
subalgebra}. It is a maximal commutative subalgebra of $U_n$, see \cite[Corollary~1]{Ov1}.

A $U_n$-module $M$ is called a {\em Gelfand-Zeitlin module} provided that the action of $\Gamma_n$
on $U_n$ is locally finite. The theory of Gelfand-Zeitlin modules originates in the papers 
\cite{DFO1,DFO2,DFO3,DFO4}, inspired by the description, due to  I.~Gelfand and M.~Zeitlin, 
of a basis in finite dimensional  $U_n$-modules consisting of $\Gamma_n$-eigenvectors in \cite{GZ1}, see also
\cite{GZ2} for a similar result for orthogonal Lie algebras. The main value of the original 
general theory  of Gelfand-Zeitlin modules 
was that it produced a family of simple $U_n$-module which depends on $n(n+1)/2$ complex parameters,
the largest known family of simple $U_n$-modules to date. This theory was generalized to orthogonal
Lie algebras in \cite{Ma3} and to quantum algebras in \cite{MT}. It was also very useful for the
study of various categories of $U_n$-modules, see \cite{Kh,Ma1,Ma4,Ma5,MO,CM,KM,MS}.

A major step in the development of the theory of Gelfand-Zeitlin modules
was made in \cite{Ov1,Ov2} (it was also put into a more general setup in 
\cite{FO}, see also \cite{FMO,FOS}) where it was, in particular, shown that all Gelfand-Zeitlin characters lift to 
$U_n$-modules and that the number of such non-isomorphic simple lifts is finite. This naturally motivated
the question of classification and explicit construction of simple Gelfand-Zeitlin modules. The main
difficulty is to construct and classify so-called {\em singular} Gelfand-Zeitlin modules, that is modules
on which the (rational) coefficients of the classical Gelfand-Zeitlin formulae have potential singularities.
A lot of progress in this direction was made recently in \cite{FGR1,FGR2,FGR3,FGR4,FGR5,FRZ1,FRZ2,GR,RZ,Vi1,Vi2,Za}
using a variety of different methods.

In the present paper we observe that the approach to construct singular simple 
Gelfand-Zeitlin modules
proposed in  \cite{Vi1,Vi2} works, with minimal adjustment, for a much larger class of algebras, 
called {\em orthogonal Gelfand-Zeitlin algebras} which were introduced in \cite{Ma2} and studied in \cite{MPT}. 

Let us now describe the results and the structure of the paper. Let $U$ be an orthogonal Gelfand-Zeitlin algebra
and $\Gamma$ its Gelfand-Zeitlin subalgebra. Our first interesting observation, presented in 
Proposition~\ref{invpol}, is that the regular action of $\Gamma$ on itself extends, via the Gelfand-Zeitlin
formulae, to a faithful action of $U$ on $\Gamma$. Using the standard adjunction argument it follows that 
the dual module $\Gamma^*$ contains each simple $\Gamma$-module as a submodule. This gives a very short 
and non-technical proof
of the original statement \cite[Theorem~2]{Ov2} on existence of Gelfand-Zeitlin $U$-modules for
arbitrary Gelfand-Zeitlin characters. Our arguments and results also work for an arbitrary orthogonal Gelfand-Zeitlin 
algebra, while \cite[Theorem~2]{Ov2} is proved just for $U_n$.

The module $\Gamma^*$ can be used to define, for each fixed Gelfand-Zeitlin character, what
we call a {\em canonical} simple Gelfand-Zeitlin module for this character, see Subsection~\ref{s3.4}.
In full generality, simple Gelfand-Zeitlin modules have not been classified. So,
it is a very unexpected feature that, for a fixed Gelfand-Zeitlin character, 
one can define a simple Gelfand-Zeitlin module in a way which does not involve any choices.

We study  $\Gamma^*$  and its Gelfand-Zeitlin submodules more closely in Section~\ref{s4}.
There we first observe that the Gelfand-Zeitlin formulae can be rewritten using divided difference
operators, see \cite{BGG,De}. 
We use divided difference operators to write down an explicit $U$-submodule  of 
$\Gamma^*$ and show, in Subsection~\ref{s4.5}, that this module is generically simple
and hence also canonical.

Section~\ref{s2} below contains all necessary preliminaries.

\textbf{Acknowledgements:} V.~M. is partially supported by
the Swedish Research Council and G{\"o}ran Gustafsson Stiftelse.
E.~V. is partially  supported by SFB TR 191, Tomsk State University,
Competitiveness Improvement Program and by Programa Institucional de
Aux{\'\i}lio {\`a} Pesquisa de Docentes Rec{\'e}m-Contratados ou Rec{\'e}m-Doutorados,
UFMG 2018.
N.~E. was partially supported by RTG grant NSF/DMS-1148634.
N.~E. thanks Victor Reiner for useful discussions.
We thank the referees for very helpful comments.

\section{Orthogonal Gelfand-Zeitlin algebras}\label{s2}

\subsection{Setup}\label{s2.1}

Let $m$ be a fixed positive integer and $\lambda=(\lambda_1,\lambda_2,\dots,\lambda_k)$ a composition of
$m$ with $k$ non-zero parts. This means that all $\lambda_i$ are positive integers and, moreover, 
$\lambda_1+\lambda_2+\dots+\lambda_k=m$. Denote by $I=I_{\lambda}$ the set of all pairs
$(i,j)$ such that $i\in\{1,2,\dots,k\}$ and $j\in\{1,2,\dots,\lambda_i\}$. Let $\Omega=\Omega_{\lambda}$ 
be the field of rational functions in $m$ variables $x_{\mathbf{a}}$, where $\mathbf{a}\in I$. 
For $\mathbf{a}=(i,j)\in I$, we will use the notation $i_{\mathbf{a}}:=i$
and $j_{\mathbf{a}}:=j$. For $i=1,2,\dots,k$, we also denote by $I^{(i)}$ the set of all 
$\mathbf{a}\in I$ for which $i_{\mathbf{a}}=i$.

For each $\mathbf{a}\in I$, let $\varphi_{\mathbf{a}}$ denote the automorphism of the field extension
$\mathbb{C}\subset \Omega$ which is uniquely defined via
\begin{displaymath}
\varphi_{\mathbf{a}}(x_{\mathbf{b}})=x_{\mathbf{b}}+\delta_{\mathbf{a},\mathbf{b}},\quad
\text{ for all } \quad \mathbf{b}\in I,
\end{displaymath}
where $\delta_{\mathbf{a},\mathbf{b}}$ is the Kronecker symbol. We denote by $\gimel$ the (abelian)
group generated by all $\varphi_{\mathbf{a}}$ with $i_{\mathbf{a}}<k$.

For each $f\in\Omega$, we have the $\mathbf{C}$-linear transformation of $\Omega$ given by 
multiplication with $f$.

\subsection{Definition}\label{s2.2}

For $i=1,2,\dots,k-1$, define the $\mathbb{C}$-linear operators $E_i$ and $F_i$ on $\Omega$ 
by the following {\em Gelfand-Zeitlin formulae}:
\begin{displaymath}
E_i:=\sum_{j=1}^{\lambda_i}   
\frac{\displaystyle \prod_{\mathbf{a}\in I^{(i+1)}}(x_{(i,j)}-x_{\mathbf{a}})}
{\displaystyle \prod_{\mathbf{b}\in I^{(i)}\setminus\{(i,j)\}}(x_{(i,j)}-x_{\mathbf{b}})} \varphi_{(i,j)}\quad
F_i:=\sum_{j=1}^{\lambda_i}   
\frac{\displaystyle \prod_{\mathbf{a}\in I^{(i-1)}}(x_{(i,j)}-x_{\mathbf{a}})}
{\displaystyle \prod_{\mathbf{b}\in I^{(i)}\setminus\{(i,j)\}}(x_{(i,j)}-x_{\mathbf{b}})} \varphi_{(i,j)}^{-1},
\end{displaymath}
where the set $I^{(0)}$ is, by convention, empty and hence the product over this set equals $1$.
For $\mathbf{a}\in I$, we also define the $\mathbb{C}$-linear operator $\gamma_{\mathbf{a}}$ on $\Omega$
as multiplication with the $j_{\mathbf{a}}$-th elementary symmetric polynomial in 
$\{x_{\mathbf{b}}:\mathbf{b}\in I^{(i_{\mathbf{a}})}\}$.

The {\em orthogonal Gelfand-Zeitlin} (OGZ)-algebra $U_{\lambda}$ associated to $\lambda$ 
is the subalgebra of the algebra of all
$\mathbb{C}$-linear transformations of $\Omega$, generated by all $E_i,F_i$, where 
$i=1,2,\dots, k-1$, and all $\gamma_{\mathbf{a}}$, where $\mathbf{a}\in I$, see \cite[Section~3]{Ma2}.

\subsection{Gelfand-Zeitlin subalgebra and Gelfand-Zeitlin modules}\label{s2.3}

The commutative subalgebra $\Gamma_{\lambda}$ of $U_{\lambda}$, generated by all 
$\gamma_{\mathbf{a}}$, where $\mathbf{a}\in I$, is called the {\em Gelfand-Zeitlin subalgebra}.
A $U_{\lambda}$-module $M$ is called a {\em Gelfand-Zeitlin} module provided that  the action of 
$\Gamma_{\lambda}$ on this module is locally finite. By \cite[Corollary~1]{Ma2}, the algebra
$\Gamma_{\lambda}$ is a {\em Harish-Chandra} subalgebra of $U_{\lambda}$ in the sense of \cite{DFO4}.

For a character $\chi:\Gamma_{\lambda}\to\mathbb{C}$ and a Gelfand-Zeitlin module $M$, we denote by
$M(\chi)$ the set of all vectors in $M$ which are annihilated by some power of the kernel of $\chi$.
Then we have
\begin{displaymath}
M=\bigoplus_{\chi}M(\chi). 
\end{displaymath}
We denote by $\mathrm{pr}_{\chi}$ the projection map $M\tto M(\chi)$ with respect to this decomposition.

\subsection{$U_n$ as an OGZ algebra}\label{s2.4}

If we take $m=n(n+1)/2$ and $\lambda=(1,2,\dots,n)$, then $U_n\cong U_{\lambda}$ such that
this isomorphism identifies $\Gamma_n$ with $\Gamma_{\lambda}$, see \cite[Section~4]{Ma2} for details.

\subsection{Group action}\label{s2.5}

Let $G=S_{\lambda_1}\times S_{\lambda_2}\times\dots\times S_{\lambda_k}$. 
Then, for $i=1,2,\dots,k$, the permutation group $S_{\lambda_i}$ permutes
$\{x_{\mathbf{a}}:\mathbf{a}\in I^{(i)}\}$ by acting on the second 
component of the pair $\mathbf{a}$. This defines an action of $G$ on $\Omega$.

Let $R$ denote the polynomial ring in all $x_{\mathbf{a}}$, where $\mathbf{a}\in I$. Then $R$
is a subring of $\Omega$ and $\Omega$ is the field of fractions of the domain $R$. 
The algebra $\Gamma_{\lambda}$ is canonically isomorphic 
to the algebra $\mathbf{R}:=R^G$ of all polynomials in 
$\{x_{\mathbf{a}}:\mathbf{a}\in I\}$ that are invariant with respect to the action of $G$.
In what follows, we will identify the algebras $\Gamma_{\lambda}$ and $\mathbf{R}$.

We also set $\underline{G}:=S_{\lambda_1}\times S_{\lambda_2}\times\dots\times S_{\lambda_{k-1}}$,
which is a subgroup of $G$ in the obvious way.

\section{Action of $U_{\lambda}$ on $\mathbf{R}$ and its consequences}\label{s3}

\subsection{Dual spaces}\label{s3.1}

Let $V$ be the $\mathbb{C}$-vector space spanned by all $x_{\mathbf{a}}$, where $\mathbf{a}\in I$.
Consider the dual vector space $V^*$ of $V$. Each
$\mathtt{v}\in V^*$ gives rise to the {\em evaluation} ring homomorphism $\mathrm{ev}_{\mathtt{v}}:R\tto \mathbb{C}$
by sending each $x_{\mathbf{a}}$, where $\mathbf{a}\in I$, to $\mathtt{v}(x_{\mathbf{a}})$.

As usual, $R^*$ denotes the vector space of all $\mathbb{C}$-linear maps from $R$ to $\mathbb{C}$.
Then the elements   $\mathrm{ev}_{\mathtt{v}}$, where $\mathtt{v}\in V^*$, 
are linearly independent elements in $R^*$. The action of $\gimel$ on $\Omega$ induces an action of
$\gimel$ on both $V^*$ and $\{\mathrm{ev}_{\mathtt{v}}\,:\,\mathtt{v}\in V^*\}$.

For $\mathtt{v}\in V^*$, we have the algebra homomorphism 
$\mathrm{\bf ev}_{\mathtt{v}}:\mathbf{R}\to \mathbb{C}$ given by
\begin{displaymath}
\mathrm{\bf ev}_{\mathtt{v}}:\mathbf{R}\hookrightarrow 
R\overset{\mathrm{ev}_{\mathtt{v}}}{\longrightarrow} \mathbb{C}. 
\end{displaymath}
The corresponding $1$-dimensional $\mathbf{R}$-module is denoted $\mathbb{C}_{\mathtt{v}}$.
Note that, for $\mathtt{v},\mathtt{w}\in V^*$, we have $\mathbb{C}_{\mathtt{v}}\cong \mathbb{C}_{\mathtt{w}}$
if and only if $\mathrm{\bf ev}_{\mathtt{v}}=\mathrm{\bf ev}_{\mathtt{w}}$ if and only if
$\mathtt{v}\in G\cdot \mathtt{w}$. In particular, $\mathbf{R}$ distinguishes $G$-orbits on $V^*$. 
Furthermore, each character of $\mathbf{R}$ is of the form
$\mathrm{\bf ev}_{\mathtt{v}}$, for some $\mathtt{v}\in V^*$. We denote by $\mathbf{m}_{\mathtt{v}}$
the kernel of $\mathrm{\bf ev}_{\mathtt{v}}$ and by $\chi_{\mathtt{v}}$ the quotient map
$\mathbf{R}\tto \mathbf{R}/\mathbf{m}_{\mathtt{v}}\cong\mathbb{C}$.

\subsection{Action  on invariant polynomials}\label{s3.2}

\begin{proposition}\label{invpol}
The algebra  $\mathbf{R}$ 
is invariant under the natural (left) action of $U_{\lambda}$.
\end{proposition}

\begin{proof}
For $f\in \mathbf{R}$, clearly $\gamma_{\mathbf{a}}\cdot f\in \mathbf{R}$, for any $\mathbf{a}\in I$. If
$i\in\{1,2,\dots,k-1\}$, then both $E_i\cdot f$ and $F_i\cdot f$ are $G$-invariant rational
function whose denominators consist of products of $x_{\mathbf{a}}-x_{\mathbf{b}}$, where
$\mathbf{a}$ and $\mathbf{b}$ are different elements of $I^{(i)}$. 

Let $g$ denote the element $E_i\cdot f$ (or $F_i\cdot f$). Let $\mathbf{a}$ and $\mathbf{b}$ 
be different elements of $I^{(i)}$. We want to prove that $x_{\mathbf{a}}-x_{\mathbf{b}}$
disappears from the denominator of $g$. Note that each denominator in the formulae for 
both $E_i$ and  $F_i$ contains the factor $x_{\mathbf{a}}-x_{\mathbf{b}}$ at most once.
Therefore, we can write
$g=\frac{h}{x_{\mathbf{a}}-x_{\mathbf{b}}}$ such that $h$ is a rational function 
in which $x_{\mathbf{a}}-x_{\mathbf{b}}$ does not appear in the denominator anymore and
which maps to $-h$ after swapping $x_{\mathbf{a}}$ and $x_{\mathbf{b}}$. The classical
description of alternating polynomials then implies that $x_{\mathbf{a}}-x_{\mathbf{b}}$
is a factor of $h$. Hence this factor cancels in the numerator and in the denominator of $g$ 
and the claim follows. 
\end{proof}

Proposition~\ref{invpol} says that the vector space $\Gamma_{\lambda}$ 
carries naturally the structure 
of a left $U_{\lambda}$-module which extends the left regular action of $\Gamma_{\lambda}$ on itself.
We refer to \cite{Ni1,Ni} for similar phenomena for some Lie algebras. Comparing with the main result
of \cite{Ni1}, it would be interesting to classify all possible $U_{\lambda}$-module structures 
on $\Gamma_{\lambda}$ extending the left regular action of $\Gamma_{\lambda}$ on itself.

We note that the $U_{\lambda}$-module $\Gamma_{\lambda}$ is not simple  
as $\Gamma_{\lambda}$ does not  have a central character. 
Indeed, the $S_{\lambda_k}$-invariant polynomials in 
$x_{(k,1)},x_{(k,2)},\dots,x_{(k,\lambda_k)}$ clearly belong to the center of $U_{\lambda}$
and the algebra of such polynomials acts freely on $\Gamma_{\lambda}$.
However, it would be interesting to know for which maximal ideals $\mathbf{m}$ in the 
center of $U_{\lambda}$ the $U_{\lambda}$-module  $\Gamma_{\lambda}/\mathbf{m}\Gamma_{\lambda}$ is simple, 
alternatively, has finite length.

\subsection{Generic regular modules}\label{s3.3}

We denote by $\mathbf{R}^*$ the set of all $\mathbb{C}$-linear maps from $\mathbf{R}$
to $\mathbb{C}$. From Proposition~\ref{invpol} we have that the space $\mathbf{R}^*$
has the natural structure of a right $U_{\lambda}$-module.

Fix some $\mathtt{v}\in V^*$ such that $\mathtt{v}(x_{\mathbf{a}})-
\mathtt{v}(x_{\mathbf{b}})\not\in\mathbb{Z}$, 
for all $i\in\{1,2,\dots,k-1\}$ and all different $\mathbf{a},\mathbf{b}\in I^{(i)}$
(we will call such $\mathtt{v}$ {\em regular}).
Denote by $M_{\mathtt{v}}$ the vector 
subspace of $\mathbf{R}^*$ generated by all elements of the form
$\mathrm{\bf ev}_{\mathtt{w}}$, where $\mathtt{w}\in\gimel\cdot \mathtt{v}$.

\begin{proposition}\label{prop41}
The space  $M_{\mathtt{v}}$ is invariant under the right $U_{\lambda}$-action.
\end{proposition}

\begin{proof}
We note that our choice of $\mathtt{v}$ ensures that  each $\mathrm{\bf ev}_{\mathtt{w}}$,
where $\mathtt{w}\in\gimel\cdot \mathtt{v}$, evaluates denominators in the formulae for all
$E_i$ and $F_i$ to non-zero elements. Therefore, directly from the definitions, it 
follows that the precomposition of any $\mathrm{\bf ev}_{\mathtt{w}}$ as above with 
any $E_i$, $F_i$ or $\gamma_{\mathbf{a}}$ stays inside $M_{\mathtt{v}}$. The claim follows.
\end{proof}

The modules described in Proposition~\ref{prop41} are the so-called {\em generic regular Gelfand-Zeitlin}
modules over $U_{\lambda}$, cf. \cite[Theorem~24]{DFO4} and \cite[Section~8]{Ma2}. 
These $U_{\lambda}$-modules are indeed 
Gelfand-Zeitlin modules as $\Gamma_{\lambda}$ acts via scalars on each $\mathrm{\bf ev}_{\mathtt{w}}$
by construction. It is easy to see that each $M_{\mathtt{v}}$ as above has finite length.
Its simple subquotients can be described using the following combinatorial object: 
\begin{itemize}
\item Identify $\gimel$, and hence also $\gimel\cdot \mathtt{v}$  with 
$\mathbb{Z}^k$, for an appropriate $k$. 
\item Consider this as the set of vertices of an unoriented graph
where edges correspond to adding vectors from the standard basis. 
\item Remove all edges which correspond to vanishing of the numerators
of Gelfand-Zeitlin formulae.
\end{itemize}
Then simple subquotients of $M_{\mathtt{v}}$ are in bijection with connected
components of the obtained (infinite) graph.

We also note that the above construction does not work if $\mathtt{v}(x_{\mathbf{a}})-
\mathtt{v}(x_{\mathbf{b}})\in\mathbb{Z}$, for some $\mathbf{a}$ and $\mathbf{b}$ in the same $I^{(i)}$, for some $i\in\{1,2,\dots,k-1\}$. Indeed, in this case
$\mathrm{\bf ev}_{\mathtt{v}}$ evaluates to zero the denominators of some of the 
coefficients in the Gelfand-Zeitlin formulae and the whole construction collapses. 

\begin{corollary}\label{corfaithful}
The $U_{\lambda}$-module $\Gamma_{\lambda}$ is faithful. 
\end{corollary}

\begin{proof}
Consider the annihilator $J$ of $\Gamma_{\lambda}=\mathbf{R}$. Then $J$ annihilates 
$\mathbf{R}^*$ as well, in particular, $J$ annihilates all generic regular
Gelfand-Zeitlin modules. Let $u\in U_{\lambda}$ be a non-zero element.
Write $u$ in the form $\displaystyle \sum_{\gamma\in\gimel}f_{\gamma}\gamma$, where
$f_{\gamma}\in\Omega$. Clearly, there exists a regular $\mathtt{v}$ such that 
$\mathrm{ev}_{\mathtt{v}}$ evaluates all non-zero $f_{\gamma}$ to non-zero complex numbers.
Since $\mathtt{v}$ is regular, all $\gamma(\mathtt{v})$ automatically belong to pairwise 
different orbits of $G$. 
This implies that the action of $u$ on $M_{\mathtt{v}}$
is non-zero. Therefore $u\not\in J$ and thus $J=0$, proving the claim. 
\end{proof}

\subsection{Existence of singular modules}\label{s3.4}

The following result generalizes \cite[Theorem~2]{Ov2} to our setup.

\begin{corollary}\label{cor37}
Any character of $\Gamma_{\lambda}$ extends to a non-zero right $U_{\lambda}$-module. 
\end{corollary}

\begin{proof}
For any $\mathtt{v}\in V^*$, the $U_{\lambda}$-submodule of the
right $U_{\lambda}$ module $\mathbf{R}^*$ generated by $\mathrm{\bf ev}_{\mathtt{v}}$ contains 
the non-zero element $\mathrm{\bf ev}_{\mathtt{v}}$. By construction, the algebra $\Gamma_{\lambda}$
acts on $\mathrm{\bf ev}_{\mathtt{v}}$ via scalars prescribed by $\mathtt{v}$. As $\Gamma_{\lambda}$
is identified with $\mathbf{R}$, varying $\mathtt{v}$ will exhaust all
characters of $\Gamma_{\lambda}$. The claim follows.
\end{proof}

From \cite[Corollary~1]{Ma2} and \cite[Subsection~1.4]{DFO4} it follows that the module constructed in 
the proof of Corollary~\ref{cor37} is, in fact, a Gelfand-Zeitlin module.

The assertion of Corollary~\ref{cor37} can be strengthened as follows.

\begin{proposition}\label{prop38}
For each $\mathtt{v}\in V^*$, the element $\mathrm{\bf ev}_{\mathtt{v}}$ is, up to a
scalar, the unique 
element of $\mathbf{R}^*$ annihilated by $\mathbf{m}_{\mathtt{v}}$.
\end{proposition}

\begin{proof}
The fact that $\Gamma_{\lambda}$ acts on $\mathrm{\bf ev}_{\mathtt{v}}$
via scalars prescribed by $\mathtt{v}$ follows directly from the definitions.
So, we just need to prove the uniqueness. Our proof follows the argument from \cite[Proposition~2]{Ni}.
Using adjunction, we compute:
\begin{displaymath}
\begin{array}{rcl}
\mathrm{Hom}_{\Gamma_{\lambda}}\big(\mathbb{C}_{\mathtt{v}},\mathbf{R}^*\big)&=&
\mathrm{Hom}_{\Gamma_{\lambda}}\big(\mathbb{C}_{\mathtt{v}},\mathrm{Hom}_{\mathbb{C}}(\mathbf{R},\mathbb{C})\big)\\
&\cong&
\mathrm{Hom}_{\mathbb{C}}\big(\mathbb{C}_{\mathtt{v}}\otimes_{\Gamma_{\lambda}}\mathbf{R},\mathbb{C}\big)
\end{array}
\end{displaymath}
and the claim of the proposition follows from the observation that $\mathbb{C}_{\mathtt{v}}\otimes_{\Gamma_{\lambda}}\mathbf{R}\cong \mathbb{C}_{\mathtt{v}}$ as $\mathbf{R}$
is a free $\Gamma_{\lambda}$-module of rank one.
\end{proof}

A similar adjunction argument shows that a simple right $U_{\lambda}$-module $M$ is a submodule of 
$\mathbf{R}^*$ if and only if $M\otimes_{U_{\lambda}}\mathbf{R}\neq 0$.

For each $\mathtt{v}\in V^*$, we denote by $N_{\mathtt{v}}$ the $U_{\lambda}$-submodule 
of $\mathbf{R}^*$ generated by 
$\mathrm{\bf ev}_{\mathtt{v}}$, that is $N_{\mathtt{v}}:=\mathrm{\bf ev}_{\mathtt{v}}\cdot U_{\lambda}$.

\begin{proposition}\label{prop81}
The module  $N_{\mathtt{v}}$ has a unique maximal submodule. 
\end{proposition}

\begin{proof}
The unique maximal submodule of  $N_{\mathtt{v}}$ is the sum of all submodules $M$ of $N_{\mathtt{v}}$
with the property  $\mathrm{\bf ev}_{\mathtt{v}}\not\in M$. The detailed proof is similar to
the proof of \cite[Proposition~7.1.11]{Di}.
\end{proof}

The quotient of $N_{\mathtt{v}}$ by its unique maximal submodule will be denoted $L_{\mathtt{v}}$ and
will be called the {\em canonical}  simple Gelfand-Zeitlin module associated to $\mathtt{v}$.
Our terminology is motivated by the fact that the construction of $L_{\mathtt{v}}$ does not use any choice and is
given purely in terms of the original Gelfand-Zeitlin formulae and dualization.
Clearly, $L_{\mathtt{v}}\cong L_{\mathtt{w}}$ if $\mathtt{v}\in G\cdot \mathtt{w}$. However, we
do not know the necessary and sufficient condition on $\mathtt{v}$ and $\mathtt{w}$ 
for $L_{\mathtt{v}}$ and $L_{\mathtt{w}}$ to be isomorphic. 

Generically, for $\mathtt{v}\in V^*$, there exists a unique simple Gelfand-Zeitlin module containing a 
non-zero element annihilated by $\mathbf{m}_{\mathtt{v}}$. This module is then automatically 
canonical. This is the case, for example, in the generic regular situation.

\section{Singular Gelfand-Zeitlin modules}\label{s4}

\subsection{Divided difference operators}\label{s4.1}

Let $\mathbf{a},\mathbf{b}\in I^{(i)}$ be an ordered pair of different elements, 
for some $i<k$, and $(\mathbf{a},\mathbf{b})$ denote the transposition 
swapping $\mathbf{a}$ and $\mathbf{b}$. Then we have, see \cite{BGG,De}, the corresponding 
{\em divided difference operator} $\partial_{\mathbf{a},\mathbf{b}}:\Omega\to \Omega$ given by
\begin{displaymath}
\partial_{\mathbf{a},\mathbf{b}}=\frac{\mathrm{id}-(\mathbf{a},\mathbf{b})}{(x_{\mathbf{a}}-x_{\mathbf{b}})}.
\end{displaymath}
The operators $\partial_{\mathbf{a},\mathbf{b}}$ satisfy the following {\em Leibniz rule}:
\begin{displaymath}
\partial_{\mathbf{a},\mathbf{b}}(fg)=\partial_{\mathbf{a},\mathbf{b}}(f)g+ 
f^{(\mathbf{a},\mathbf{b})}\partial_{\mathbf{a},\mathbf{b}}(g),
\end{displaymath}
where $f^{(\mathbf{a},\mathbf{b})}$ denotes the outcome of the action of $(\mathbf{a},\mathbf{b})$ on $f$.
This Leibniz rule implies the following variation which should be understood
as an equality of operators acting on $\Omega$, where $\gamma\in\gimel$,
\begin{displaymath}
\partial_{\mathbf{a},\mathbf{b}}\circ f\circ\gamma=\partial_{\mathbf{a},\mathbf{b}}(f)\circ\gamma+ 
f^{(\mathbf{a},\mathbf{b})}\circ\partial_{\mathbf{a},\mathbf{b}}\circ\gamma. 
\end{displaymath}
We also have $\partial_{\mathbf{a},\mathbf{b}}=-\partial_{\mathbf{b},\mathbf{a}}$. 

For a fixed linear order $\prec$ on $I^{(i)}$, we set
\begin{displaymath}
\partial_{(\mathbf{a},\mathbf{b})}=
\begin{cases}
\partial_{\mathbf{a},\mathbf{b}}, &  \mathbf{a}\prec \mathbf{b};\\
\partial_{\mathbf{b},\mathbf{a}}, &  \mathbf{b}\prec \mathbf{a}.
\end{cases}
\end{displaymath}
The order $\prec$ allows us to view $G$ as a Coxeter group such that simple reflections are given by 
those transpositions which swap neighboring elements with respect to $\prec$.
Then the divided difference operators which correspond 
to simple reflections satisfy the defining relations of the nil-Coxeter algebra. In particular, 
the dimension of the algebra which such operators generate coincides with the 
cardinality of $G$. Furthermore, to each $w\in G$ 
with a fixed reduced expression $w=s_1s_2\cdots s_l$, we can associate the element
$\partial_w=\partial_{s_1}\partial_{s_2}\cdots \partial_{s_l}$ and this element will not depend
on the reduced expression, see \cite[Theorem~3.4~b)]{BGG}. 
If the expression is not reduced, then
$\partial_w=0$, see \cite[Theorem~3.4~a)]{BGG}.
The elements $\partial_w$, for $w\in G$,
are linearly independent as operators on $R$.
We define the {\em degree} of $\partial_w$ as the length of $w$.
A similar construction works for any parabolic subgroup of $G$ viewed as a Coxeter group 
in the natural way.

For $u\in G$ and simple reflection $s$, we set $\partial^{(u)}_{s}=u\partial_{s}u^{-1}$
and, for $w\in G$, define $\partial^{(u)}_w=\partial^{(u)}_{s_1}\partial^{(u)}_{s_2}\cdots \partial^{(u)}_{s_l}$,
where $w=s_1s_2\cdots s_l$ is a reduced expression.

\subsection{Realization of $U_{\lambda}$ via divided difference operators}\label{s4.2}
The results of this subsection are partially inspired by \cite{Early1}.

For a fixed element $i\in\{1,2,\dots,k-1\}$, set, for simplicity, $m=\lambda_i$. Let
$\prec$ be a fixed linear order on $\{1,2,\dots,m\}=\{a_1,a_2,\dots,a_m\}$, where
$a_{1}\prec a_{2}\prec \dots\prec a_{m}$. Let $\mu$ be a composition of $m$ with 
non-zero parts. We identify $\mu$ with the following decomposition of $\{1,2,\dots,m\}$ into subsets:
\begin{displaymath}
\{a_1,a_2,\dots,a_{\mu_1}\},\quad\{a_{\mu_1+1},a_{\mu_1+2},\dots,a_{\mu_1+\mu_2}\},\quad\dots\quad. 
\end{displaymath}
Let 
\begin{displaymath}
\underline{\mu_j}:=\{a_{\mu_1+\mu_2+\dots+\mu_{j-1}+1},a_{\mu_1+\mu_2+\dots+\mu_{j-1}+2},\dots,
a_{\mu_1+\mu_2+\dots+\mu_{j}}\} 
\end{displaymath}
denote the subset corresponding to the part $\mu_j$ (here $\mu_0=0$ by convention).

For a part $\mu_j$, we set $\min(\mu_j)=\mu_1+\mu_2+\dots+\mu_{j-1}+1$. This is the index of the minimal 
element  in $\underline{\mu_j}$ with respect to $\prec$. If $\mu_j=1$, we define
$\partial(\mu,j)$ to be the identity operator. If $\mu_j>1$, we define 
$\partial(\mu,j)$ as the operator
\begin{displaymath}
\partial_{(a_{\min(\mu_j)+\mu_j-2},a_{\min(\mu_j)+\mu_j-1})}\circ\cdots\circ
\partial_{(a_{\min(\mu_j)+1},a_{\min(\mu_j)+2})}\circ\partial_{(a_{\min(\mu_j)},a_{\min(\mu_j)+1})}. 
\end{displaymath}
Finally, we denote by  $f(\mu,j)^{\pm}$ the following rational function:
\begin{displaymath}
f(\mu,j)^{\pm}:=\frac{\displaystyle \prod_{\mathbf{a}\in I^{(i\pm 1)}}(x_{(i,a_{\min(\mu_j)})}-x_{\mathbf{a}})}{
\displaystyle \prod_{\mathbf{b}\in I^{(i)}\setminus\underline{\mu_j}}(x_{(i,a_{\min(\mu_j)})}-x_{\mathbf{b}})}.
\end{displaymath}

\begin{proposition}\label{prop73}
The restriction of the action of the
generators $E_i$ and $F_i$ of $U_{\lambda}$ from $\Omega$ to $\mathbf{R}$ 
is given by the following operators:
\begin{gather*}
E_i=\sum_j \partial(\mu,j)f(\mu,j)^{+}\varphi_{(i,\min(\mu_j))},\\
F_i=\sum_j \partial(\mu,j)f(\mu,j)^{-}\varphi_{(i,\min(\mu_j))}^{-1}.
\end{gather*}
\end{proposition}

\begin{proof}
Let us try to rewrite the Gelfand-Zeitlin formulae which describe the
action of $E_i$ when applied to elements in $\mathbf{R}$ (for $F_i$ one can use
similar arguments). Let $s$ be some element in $\{1,2,\dots,m\}$.
We need to check that the coefficient at $\varphi_{(i,s)}$ on the right hand side of the
formula in the formulation is the correct one. We use induction on $\lambda_i$ via the size of the part
$\mu_j$ of $\mu$ containing $s$. If $\mu_j=1$, then the claim follows directly from  
the Gelfand-Zeitlin formulae. In the inductive procedure below we will prove the result
at the same time for all $s\in \underline{\mu_j}$.

To prove the induction step, we assume that $\mu_j>1$ and set 
\begin{displaymath}
\mathbf{c}:=(i,a_{\mu_1+\mu_2+\dots+\mu_{j}-1}),\quad
\mathbf{d}:=(i,a_{\mu_1+\mu_2+\dots+\mu_{j}}).
\end{displaymath}
Note that  $a_{\mu_1+\mu_2+\dots+\mu_{j}}$ is the maximum element of $\underline{\mu_j}$ with respect to $\prec$ and
$a_{\mu_1+\mu_2+\dots+\mu_{j}-1}$ 
is the maximum element of $\underline{\mu_j}\setminus\{a_{\mu_1+\mu_2+\dots+\mu_{j}}\}$ with respect to $\prec$.
We now use the induction hypothesis and compute: 
\begin{displaymath}
\frac{\mathrm{id}-(\mathbf{c},\mathbf{d})}{x_{\mathbf{c}}-x_{\mathbf{d}}}\left(
\sum_{\mathbf{u}\in \underline{\mu_j}\setminus\{\mathbf{d}\}}^{}   
\frac{\displaystyle \prod_{\mathbf{a}\in I^{(i+1)}}(x_{\mathbf{u}}-x_{\mathbf{a}})}
{\displaystyle \prod_{\mathbf{b}\in I^{(i)}\setminus \{\mathbf{u},\mathbf{d}\}}
(x_{\mathbf{u}}-x_{\mathbf{b}})} \varphi_{\mathbf{u}}
\right).
\end{displaymath}
The coefficients at $\varphi_{\mathbf{c}}$ and $\varphi_{\mathbf{d}}$ are, clearly, correct. 
For $\mathbf{u}\in I^{(i)}\setminus \{\mathbf{c},\mathbf{d}\}$, the coefficient at $\varphi_{\mathbf{u}}$ is
\begin{displaymath}
\frac{1}{(x_{\mathbf{c}}-x_{\mathbf{d}})}\left(
\frac{\displaystyle \prod_{\mathbf{a}\in I^{(i+1)}}(x_{\mathbf{u}}-x_{\mathbf{a}})}
{\displaystyle \prod_{\mathbf{b}\in I^{(i)}\setminus \{\mathbf{u},\mathbf{d}\}}(x_{\mathbf{u}}-x_{\mathbf{b}})}- 
\frac{\displaystyle \prod_{\mathbf{a}\in I^{(i+1)}}(x_{\mathbf{u}}-x_{\mathbf{a}})}
{(x_{\mathbf{u}}-x_{\mathbf{d}})\displaystyle \prod_{\mathbf{b}\in I^{(i)}\setminus \{\mathbf{u},\mathbf{c},\mathbf{d}\}}
(x_{\mathbf{u}}-x_{\mathbf{b}})}
\right).
\end{displaymath}
Setting $f$ to be the common numerator and $g$ to be the product expression in the denominator of the last
fraction, we have
\begin{multline*}
\frac{1}{(x_{\mathbf{c}}-x_{\mathbf{d}})}
\left(\frac{f}{(x_{\mathbf{u}}-x_{\mathbf{c}})g}-\frac{f}{(x_{\mathbf{u}}-x_{\mathbf{d}})g}\right)=\\=
\frac{f}{g}
\left(\frac{1}{(x_{\mathbf{c}}-x_{\mathbf{d}})(x_{\mathbf{u}}-x_{\mathbf{c}})}-
\frac{1}{(x_{\mathbf{c}}-x_{\mathbf{d}})(x_{\mathbf{u}}-x_{\mathbf{d}})}\right)=\\=
\frac{f}{g}\cdot
\frac{1}{(x_{\mathbf{u}}-x_{\mathbf{c}})(x_{\mathbf{u}}-x_{\mathbf{d}})}=
\frac{\displaystyle \prod_{\mathbf{a}\in I^{(i+1)}}(x_{\mathbf{u}}-x_{\mathbf{a}})}
{\displaystyle \prod_{\mathbf{b}\in I^{(i)}\setminus \{\mathbf{u}\}}(x_{\mathbf{u}}-x_{\mathbf{b}})}.
\end{multline*}
This shows that the coefficient at $\varphi_{\mathbf{u}}$ is also correct.
This proves the claim.
\end{proof}

Note that, varying $\prec$, we obtain different ways to write down 
the action of $E_i$ and $F_i$ on  $\mathbf{R}$.
The assertion of Proposition~\ref{prop73} seems a priori very surprising as
the intermediate computations in the proof necessarily involve elements
which are not $G$-invariant.

\subsection{Construction of singular Gelfand-Zeitlin modules}\label{s4.4}

For $\mathtt{v}\in V^*$, we denote by $G_{\mathtt{v}}$ the stabilizer of $\mathtt{v}$ in $\underline{G}$.
For $\rho\in G_{\mathtt{v}}$ and $Q$ an element of a set on which $G_{\mathtt{v}}$ acts, we will denote
by $Q^{\rho}$ the outcome of the action of $\rho$ on $Q$.

Fix $\mathtt{v}\in V^*$ with the following properties:
\begin{itemize}
\item The group $G_{\mathtt{v}}$ is the maximum element, with respect to 
inclusions, in the set $\{G_{\mathtt{w}}:\mathtt{w}\in\gimel\cdot \mathtt{v}\}$.
\item Every orbit of $G_{\mathtt{v}}$ in each $I^{(i)}$ has the form
\begin{displaymath}
\{(i,p),(i,p+1),(i,p+2),\dots,(i,p+q)\}, 
\end{displaymath}
for some $p,q$.
\end{itemize}
The set $\gimel\cdot\mathtt{v}$
is invariant under the action of $G_{\mathtt{v}}$ and we can write it as a disjoint union of orbits:
\begin{displaymath}
\gimel\cdot\mathtt{v}=\coprod_{j\in J} \mathcal{O}_j,
\end{displaymath}
where $J$ is just some indexing set. For each element $j\in J$, we fix the unique 
representative $\mathtt{u}_j$ in the orbit 
$\mathcal{O}_j$ such that the following condition is satisfied:
If $(i,p)$ and $(i,p+1)$ belong to the same orbit of $G_{\mathtt{v}}$, for some $i$ and $p$,
then 
\begin{equation}\label{eqneq5}
\mathbb{Z}\ni(\mathtt{u}_j-\mathtt{v})(x_{(i,p)})\geq (\mathtt{u}_j-\mathtt{v})(x_{(i,p+1)})\in\mathbb{Z}. 
\end{equation}
We have $\mathtt{u}_j=\xi_j^{-1}(\mathtt{v})$, for some $\xi_j\in\gimel$.
Let $X_j$ denote the set of shortest representatives of cosets in $G_{\mathtt{v}}/G_{\mathtt{u}_j}$.

We fix the natural Coxeter presentation of $G$, i.e. the one where all simple reflections
have the form $((i,p),(i,p+1))$, for some $i$ and $p$. Consider the following two subsets in $\mathbf{R}^*$:
\begin{equation}\label{eq81}
\mathbf{B}:=\bigcup_{j\in J}\bigcup_{w\in X_j}\{\mathrm{ev}_{\mathtt{v}}\circ \partial_w\circ 
\xi_j\}\qquad\text{ and }\qquad
\underline{\mathbf{B}}:=\bigcup_{j\in J}\bigcup_{w\in 
G_{\mathtt{v}}}\{\mathrm{ev}_{\mathtt{v}}\circ \partial_w\circ 
\xi_j\}.
\end{equation}
We note the use of $\mathrm{ev}_{\mathtt{v}}$ instead of $\mathrm{\bf ev}_{\mathtt{v}}$ here due
to the presence of the shifts $\xi_j$.

\begin{lemma}\label{independencelemma}
For any $\rho\in G_{\mathtt{v}}$ and any 
$\mathrm{ev}_{\mathtt{v}}\circ \partial_w\circ \xi_j\in \underline{\mathbf{B}}$, we have
\begin{displaymath}
\mathrm{ev}_{\mathtt{v}}\circ \partial_w\circ \xi_j=
\mathrm{ev}_{\mathtt{v}}\circ \partial^{(\rho)}_{w}\circ \xi_j^{\rho},
\end{displaymath}
as elements of $\mathbf{R}^{*}$.
\end{lemma}

\begin{proof}
As $\rho\in G_{\mathtt{v}}$, we have $\mathrm{ev}_{\mathtt{v}}=\mathrm{ev}_{\mathtt{v}}\circ\rho$.
We have  $\rho\circ \partial_w\circ \rho^{-1}=\partial^{(\rho)}_{w}$ by \cite[Lemma~3.3]{BGG}
and hence 
\begin{displaymath}
\mathrm{ev}_{\mathtt{v}}\circ \partial_w\circ \xi_j=
\mathrm{ev}_{\mathtt{v}}\circ \partial^{(\rho)}_{w}\circ \xi_j^{\rho}\circ \rho^{-1}.
\end{displaymath}
As the set $\mathbf{R}$ consists of $G$-invariant polynomials, the claim follows.
\end{proof}

\begin{lemma}\label{lem82}
The element $\mathrm{ev}_{\mathtt{v}}\circ \partial_w\circ  \xi_j$ of $\underline{\mathbf{B}}$
is zero (as an element of $\mathbf{R}^{*}$) provided that $w\not\in X_j$.
\end{lemma}

\begin{proof}
If $w\not\in X_j$, we may write $\partial_w=\partial_{w'}\partial_{\tau}$ for some simple reflection
$\tau\in G_{\mathtt{u}_j}$. For any $f\in\mathbf{R}$, we then have that $\tau(\xi_j(f))=
\xi_j(f)$. Therefore $\partial_{\tau}\xi_j(f)=0$. The claim follows.
\end{proof}

From Lemma~\ref{lem82} we get that either $\underline{\mathbf{B}}=\mathbf{B}$ or
$\underline{\mathbf{B}}=\mathbf{B}\cup\{0\}$.
Let $M_{\mathtt{v}}$ be the linear span, in $\mathbf{R}^*$, of $\underline{\mathbf{B}}$.
Then $M_{\mathtt{v}}$ coincides with the linear span, in $\mathbf{R}^*$, of ${\mathbf{B}}$.

\begin{theorem}\label{thm81}
The space $M_{\mathtt{v}}$ is closed under the action of $U_{\lambda}$, moreover, 
${\mathbf{B}}$ is a basis in $M_{\mathtt{v}}$.
\end{theorem}

\begin{proof}
To prove that $M_{\mathtt{v}}$ is  closed under the action of $U_{\lambda}$, let us check that it is closed
under the right multiplication with the generators. Multiplying on the right with an element 
$h\in \mathbf{R}$, we can move this element past  $\xi_j$ by acting on it, that is 
\begin{displaymath}
\xi_j\circ h= h^{\xi_j} \circ \xi_j,
\end{displaymath}
and then we can move the new
element $h^{\xi_j} $ past $\partial_w$ using the Leibniz rule. 
When we reach $\mathrm{ev}_{\mathtt{v}}$, we evaluate
at the point $\mathtt{v}$. 
This means that $M_{\mathtt{v}}$ is closed under the action of
all $\gamma_{\mathbf{a}}$, where $\mathbf{a}\in I$.

Let us prove that $M_{\mathtt{v}}$ is closed under the action of $E_i$ and $F_i$, where $i<k$.
We prove this for $E_i$ and in case of $F_i$ the proof is similar.
Consider the element
\begin{equation}\label{eq3}
\mathrm{ev}_{\mathtt{v}}\circ \partial_w\circ  \xi_j\circ E_i, 
\end{equation}
for some $\mathrm{ev}_{\mathtt{v}}\circ \partial_w\circ  \xi_j\in\mathbf{B}$. We want to prove
that \eqref{eq3} can be written as a linear combination of elements in $\underline{\mathbf{B}}$.

Let $\mu$ be the composition of $\{1,2,\dots,\lambda_i\}$ corresponding to the orbits of 
$G_{\mathtt{u}_j}$ on $\{1,2,\dots,\lambda_i\}$. We write $E_i$ in the form given by 
Proposition~\ref{prop73}, for the composition $\mu$. 
Then every divided difference operator $\partial_{\mathbf{a},\mathbf{b}}$
appearing in this expression has the property that $\xi_j$ is invariant under 
$(\mathbf{a},\mathbf{b})$. In this case the application of
$\xi_j$ to $x_{\mathbf{a}}-x_{\mathbf{b}}$ gives $x_{\mathbf{a}}-x_{\mathbf{b}}$.
This implies the relation
\begin{displaymath}
\xi_j\circ \partial_{\mathbf{a},\mathbf{b}}=\partial_{\mathbf{a},\mathbf{b}}\circ \xi_j
\end{displaymath}
which allows us to move all operators $\partial_{\mathbf{a},\mathbf{b}}$ to the left of $\xi_j$.

For each part $\mu_s$ of $\mu$, the corresponding rational function $f(\mu,s)^+$ has, by construction, the 
property that, for any $w\in G_{\mathtt{u}_j}$, the evaluation of $(\xi_j(f(\mu,s)^+))^w$ at
$\mathtt{v}$ is well-defined as none of the denominators of $(\xi_j(f(\mu,s)^+))^w$ evaluates to zero.
Therefore we may use the Leibniz rule to move $(\xi_j(f(\mu,s)^+))^w$ to the left past all
the divided difference operators which appear in our expression and then evaluate the
resulting function at $\mathtt{v}$. Note that the composition $\xi_j\varphi_{(i,s)}$  may be
an  element in $\gimel$ which does not coincide with any $\xi_{j'}$.
In the latter case we may apply  Lemma~\ref{independencelemma} and, finally, get a linear combination
of elements in $\underline{\mathbf{B}}$. This completes the proof
of the first statement.

It remains to prove that $\mathbf{B}$ is a basis. 
Recall that the divided difference operators 
annihilate all symmetric polynomials (note that the latter are identified with
$\Gamma_{\lambda}$). Applying the shift of
variables corresponding to going from $\mathtt{v}$ to $\mathtt{u}_j$
to all symmetric polynomials, we obtain a vector space.
Consider the image of this vector space in the
coinvariant algebra for $G_{\mathtt{v}}$.
By \cite[Endomorphismensatz~3(i)]{So}, this image 
coincides with the algebra of $G_{\xi_j}$-invariants in the coinvariant algebra 
for $G_{\mathtt{v}}$. 

Let $v$ be a non-zero element in the socle of
the algebra of $G_{\xi_j}$-invariants in the coinvariant algebra 
for $G_{\mathtt{v}}$.
By \cite[Theorems~5.5]{BGG}, the images of $v$ under the action of 
the divided difference  operators from $X_j$ form 
a linearly independent system of elements of the coinvariant algebra. 
This implies that the part of our set $\mathbf{B}$ which corresponds 
to a given  Gelfand-Zeitlin character is linearly independent. As  
elements corresponding to different Gelfand-Zeitlin characters are
obviously linearly independent, we obtain that all elements in 
$\mathbf{B}$ are linearly independent. The claim follows.
\end{proof}

The map $w\mapsto w(\xi_j)$ is a bijection from $X_j$ to $\mathcal{O}_j$.
Therefore the elements of the basis $\mathbf{B}$ of  $M_{\mathtt{v}}$ are in a natural 
bijection with the  elements  in $\gimel\cdot \mathtt{v}$.

From \cite[Endomorphismensatz~3(i)]{So} it even follows that the space 
$M_{\mathtt{v}}(\chi_{\mathtt{w}})$, considered as a module over 
$\Gamma_{\lambda}/\mathrm{Ann}_{\Gamma_{\lambda}}
(M_{\mathtt{v}}(\chi_{\mathtt{w}}))$, 
is isomorphic to the regular representation of the algebra of $G_{\mathtt{w}}$-invariants in the 
coinvariant algebra of $G_{\mathtt{v}}$.

We expect that each simple Gelfand-Zeitlin $U_{\lambda}$-module 
is a subquotient of $M_{\mathtt{v}}$, for some ${\mathtt{v}}$.

\subsection{Sufficient conditions for simplicity}\label{s4.5}

\begin{theorem}\label{thm77}
Let $\mathtt{v}\in V^*$ be as in  Subsection~\ref{s4.4}. Further, assume that, for any $i=1,2,\dots,k-1$ 
and for any $\mathbf{a}\in I^{(i)}$ and $\mathbf{b}\in I^{(i+1)}$, we have
$\mathtt{v}(x_{\mathbf{a}})-\mathtt{v}(x_{\mathbf{b}})\not\in \mathbb{Z}$.
Then the module $M_{\mathtt{v}}$ is simple. Consequently, $M_{\mathtt{v}}\cong L_{\mathtt{v}}$ in this case.
\end{theorem}

\begin{proof}
{\bf Step~1.} We start by proving that any submodule of $M_{\mathtt{v}}$
contains all $\mathrm{\bf ev}_{\mathtt{w}}$, where $\mathtt{w}\in\gimel\cdot\mathtt{v}$

For this, we first prove the following two statements:
\begin{itemize}
\item For any $\mathtt{w}\in\gimel\cdot \mathtt{v}$, any generator $E_i$ and any 
$j\in\{1,2,\dots,\lambda_i\}$, the vector $\mathrm{pr}_{\chi_{\varphi_{(i,j)}(\mathtt{w})}}
(\mathrm{\bf ev}_{\mathtt{w}}\circ E_i)$ is non-zero.
\item For any $\mathtt{w}\in\gimel\cdot\mathtt{v}$, any generator $F_i$ and any 
$j\in\{1,2,\dots,\lambda_i\}$, the vector $\mathrm{pr}_{\chi_{\varphi_{(i,j)}^{-1}(\mathtt{w})}}
(\mathrm{\bf ev}_{\mathtt{w}}\circ F_i)$ is non-zero.
\end{itemize}
We will prove the first statement and the second one is proved similarly.
Without loss of generality we may assume that $\mathtt{w}=\mathtt{u}_s$, for some $s\in J$ 
(cf. Subsection~\ref{s4.4}).

Let $\mu$ be the composition of $\lambda_i$ corresponding to the orbits of $G_{\mathtt{w}}$ on $I^{(i)}$.
Write $E_i$ in the form given by Proposition~\ref{prop73} 
with respect to the composition $\mu$ and consider the element
$\mathrm{\bf ev}_{\mathtt{w}}\circ E_i$. Note that, if we vary $j$ along some orbit of 
$G_{\mathtt{w}}$ on $I^{(i)}$, this does not affect the Gelfand-Zeitlin character 
$\chi_{\varphi_{(i,j)}(\mathtt{w})}$. We have the obvious bijection between 
orbits of $G_{\mathtt{w}}$ on $I^{(i)}$ and different Gelfand-Zeitlin characters 
of the form $\chi_{\varphi_{(i,j)}(\mathtt{w})}$. Therefore in what follows we may assume that
$j=\min(\mu_s)$, for some part $\mu_s$ of $\mu$.

The summands of $E_i$ which correspond to the parts different from $\mu_s$ do not effect 
$\mathrm{pr}_{\chi_{\varphi_{(i,j)}(\mathtt{w})}}(\mathrm{\bf ev}_{\mathtt{w}}\circ E_i)$
and thus we only need to consider the summand corresponding to $\mu_s$. 
Let $a_1<a_2<\dots<a_{\mu_s}$ be elements of $\underline{\mu_s}$. For
$t=1,2,\dots,\mu_s-1$, set
\begin{displaymath}
w_t=(a_t,a_{t+1})\dots(a_2,a_3)(a_1,a_2). 
\end{displaymath}
Using the Leibniz rule, we can write 
\begin{multline*}
\partial(\mu,s)\circ f(\mu,s)^+=
w_{\mu_s-1}(f(\mu,s)^+)\circ\partial_{a_{\mu_s-1},a_{\mu_s}}\circ\cdots\circ\partial_{a_2,a_3}\circ\partial_{a_1,a_2}+\\+
\partial_{a_{\mu_s-1},a_{\mu_s}}(w_{\mu_s-2}(f(\mu,s)^+))\circ
\partial_{a_{\mu_s-2},a_{\mu_s-1}}\circ\cdots\circ\partial_{a_1,a_2}+\dots\\\dots+
\partial_{a_{\mu_s-1},a_{\mu_s}}\circ\cdots\circ\partial_{a_2,a_3}\circ\partial_{a_1,a_2}(f(\mu,s)^+).
\end{multline*}
As $\mathrm{\bf ev}_{\mathtt{w}}(w_{\mu_s-1}(f(\mu,s)^+))\neq 0$, due to our
assumptions on $\mathtt{v}$, we see that the element $\mathrm{\bf ev}_{\mathtt{w}}\circ E_i$ 
has a non-zero coefficient at the  element
$\mathrm{\bf ev}_{\mathtt{w}}\circ\partial(\mu,s)\circ\varphi_{(i,\min(\mu_s))}$
and the rest is a linear combination of divided difference operators of strictly smaller degree.
This implies that $\mathrm{pr}_{\chi_{\varphi_{(i,j)}(\mathtt{w})}}
(\mathrm{\bf ev}_{\mathtt{w}}\circ E_i)$ is non-zero.

The above implies that every submodule of $M_{\mathtt{v}}$ contains 
$\mathrm{\bf ev}_{\mathtt{v}}$. Set $N:=\mathrm{\bf ev}_{\mathtt{v}}\circ U_{\lambda}$.
In combination with Proposition~\ref{prop38}, the above implies that $N$ 
contains all $\mathrm{\bf ev}_{\mathtt{w}}$, where $\mathtt{w}\in\gimel\cdot\mathtt{v}$.
Also, note that, taking Step~1 into account, 
the statement of the theorem is equivalent to showing that 
$N=M_{\mathtt{v}}$. 

{\bf Step~2.} Next we claim the following: if $N$ contains some element that, written in the basis 
$\mathbf{B}$, has a non-zero coefficient at $\mathbf{ev}_{\mathtt{v}}\circ \partial_w\circ  \xi$
where $\partial_w$ has maximal possible degree (i.e. $w$ has maximal possible length) among
all elements of the form $\mathbf{ev}_{\mathtt{v}}\circ \partial_x\circ  \xi$ in $\mathbf{B}$,
then $N$ contains the whole subspace $M_{\mathtt{v}}(\chi_{\mathtt{u}})$.

Consider the Gelfand-Zeitlin subspace $M_{\mathtt{v}}(\chi_{\mathtt{u}})$ 
of $M_{\mathtt{v}}$
containing such $\mathbf{ev}_{\mathtt{v}}\circ \partial_w\circ  \xi$.
As mentioned in the previous subsection, 
as a module over $\Gamma_{\lambda}$, this subspace is isomorphic
to the regular module over the (local) algebra of 
$G_{\mathtt{u}}$-invariants in the coinvariant algebra of 
$G_{\mathtt{v}}$. In particular, the $\Gamma_{\lambda}$ 
module $M_{\mathtt{v}}(\chi_{\mathtt{u}})$ has a unique maximal submodule.
Because of the maximality of the degree for
$\mathbf{ev}_{\mathtt{v}}\circ \partial_w\circ  \xi$, this element
cannot be written as a linear combination of elements of the form 
$\mathbf{ev}_{\mathtt{v}}\circ \partial_{w'}\circ  \xi_j\circ u$,
where $u\in\Gamma_{\lambda}$ 
and the length of $w'$ is strictly smaller than the length of $w$. Consequently,
the element $\mathbf{ev}_{\mathtt{v}}\circ \partial_w\circ  \xi$ generates
$M_{\mathtt{v}}(\chi_{\mathtt{u}})$ as a $\Gamma_{\lambda}$-module.
This implies the claim and, in particular, shows that, 
if $N$ contains some element that, written in the basis $\mathbf{B}$, 
has a non-zero coefficient at $\mathbf{ev}_{\mathtt{v}}\circ \partial_w\circ  \xi$
where $\partial_w$ has maximal possible degree, then 
$N$ contains $\mathbf{ev}_{\mathtt{v}}\circ \partial_w\circ  \xi$.

{\bf Step~3.}
It remains to show that, for each $\mathbf{ev}_{\mathtt{v}}\circ \partial_w\circ  \xi$
where $\partial_w$ has maximal possible degree, the module $N$ contains some element 
which, when written in the basis $\mathbf{B}$, 
has a non-zero coefficient at $\mathbf{ev}_{\mathtt{v}}\circ \partial_w\circ  \xi$.
Note that the element $\mathbf{ev}_{\mathtt{v}}\in N$ is itself of such form.
Note also that, if $\mathbf{ev}_{\mathtt{v}}\circ \partial_w\circ  \xi$ is such an element,
then $w$ is the longest element among all shortest coset representatives in 
$G_{\mathtt{v}}/G_{\mathtt{w}}$,  where  $\mathtt{w}$ is the shift of $\mathtt{v}$ by $\xi$. 
If $G_{\mathtt{w}}=G_{\mathtt{v}}$, the claim of Step~3 follows directly from Step~1. 

Now we would like to establish two reduction procedures to prove the claim of Step~3.
Let $\mathtt{w}$ be the shift of $\mathtt{v}$ by $\xi$. We will use 
the first procedure to change $\mathtt{w}$
to some $\mathtt{w}'$ such that $G_{\mathtt{w}'}\subset G_{\mathtt{w}}$ 
and the second one to change $\mathtt{w}$
to some $\mathtt{w}'$ such that
$G_{\mathtt{w}}\subset G_{\mathtt{w}'}$.
As we will see, in the case $G_{\mathtt{w}}=G_{\mathtt{w}'}$, both procedures 
lead to the same outcome.

Assume first that $N$ contains 
some $\mathbf{ev}_{\mathtt{v}}\circ \partial_w\circ  \xi$ with 
$\partial_w$ of maximal possible degree.
Let $\mathtt{w}$ be the shift of $\mathtt{v}$ by $\xi$. Consider
the element $\mathbf{ev}_{\mathtt{v}}\circ \partial_w\circ  \xi\circ E_i$
and assume that there is some $j$ such that, for $\xi':=\varphi_{(i,j)}\xi$
and $\mathtt{w}'$ being the shift of $\mathtt{v}$ by $\xi'$,
we have $G_{\mathtt{w}'}\subset G_{\mathtt{w}}$. We claim that 
$N$ contains  $\mathbf{ev}_{\mathtt{v}}\circ \partial_{w'}\circ  \xi'$ where
$w'$ is the longest elements among the shortest coset representatives
in $G_{\mathtt{v}}/G_{\mathtt{w}'}$.
Ignoring the contributions of those components in $\mathtt{w}$ and $\mathtt{w}'$ that are identical,
we may assume that $G_{\mathtt{w}}=S_a$, where $a\geq 1$, and $G_{\mathtt{w}'}=S_1\times S_{a-1}$.
If $a=1$, then $G_{\mathtt{w}}=G_{\mathtt{w}'}$ and the claim follows
from the computation similar to the one used in Step~1, so we assume $a>1$.
We use Proposition~\ref{prop73} to write $E_i$ in the form
$\partial(\mu,1)f(\mu,1)^+\varphi_{(i,1)}$ where $\mu$ is just one block (which is
the orbit of our $S_a$). Then 
$\xi$ commutes with $\partial(\mu,1)$ and we can use, as above, the Leibniz rule to 
move the polynomial $f(\mu,1)^+$
all the way to the left and evaluate it to a non-zero number using our assumptions.
Taking the last sentence in Step~2 into account, 
this gives us the element $\partial_w\partial(\mu,1)$ which 
is exactly the element of maximal possible degree
that we needed. Similarly one considers the case of $F_i$
and $\xi':=\varphi_{(i,j)}^{-1}\xi$.

For the second procedure, we assume that $N$ contains 
some $\mathbf{ev}_{\mathtt{v}}\circ \partial_w\circ  \xi$ with 
$\partial_w$ of maximal possible degree.
Let $\mathtt{w}$ be the shift of $\mathtt{v}$ by $\xi$. Consider
the element $\mathbf{ev}_{\mathtt{v}}\circ \partial_w\circ  \xi\circ E_i$
and assume that there is some index $j$ in row $i$ that is fixed by $G_{\mathtt{w}}$.
Let $\xi':=\varphi_{(i,j)}\xi$ and $\mathtt{w}'$
be the shift of $\mathtt{v}$ by $\xi'$. We claim that 
$N$ contains  $\mathbf{ev}_{\mathtt{v}}\circ \partial_{w'}\circ  \xi'$ where
$w'$ is the longest elements among the shortest coset representatives
in $G_{\mathtt{v}}/G_{\mathtt{w}'}$. Note that in this situation
$G_{\mathtt{w}'}$ contains $G_{\mathtt{w}}$, in particular,
$w'$ is a shortest coset representative in $G_{\mathtt{v}}/G_{\mathtt{w}}$.
Since $N$ contains $\mathbf{ev}_{\mathtt{v}}\circ \partial_w\circ  \xi$,
from Step~2 it follows that
$N$ also contains $\mathbf{ev}_{\mathtt{v}}\circ \partial_{w'}\circ  \xi$.
Using the computation similar to the one used in Step~1 and the last sentence
in Step~2,
one shows that $N$ contains $\mathbf{ev}_{\mathtt{v}}\circ \partial_{w'}\circ  \xi'$.
Similarly one considers the case of $F_i$
and $\xi':=\varphi_{(i,j)}^{-1}\xi$.

Finally, we note that any element in $\gimel\cdot\mathtt{v}$ can be obtained from $\mathtt{v}$ by an 
appropriate sequence of applications of these two reduction procedures.
Indeed, given some vector $\mathbf{x}$ in some $\mathbb{Z}^r$, we can 
start from the zero vector and use the first reduction to
increase the first coefficient from $0$ to $r$, then the second coefficient
from $0$ to $r-1$ and so on. Thus, using the first reduction, we will
get, from the zero vector, a vector with all different coefficients.
Now we can use the first reduction to increase these different coefficients 
such that all coefficients in the outcome are different and bigger than all
coefficients of $\mathbf{x}$. Next we can use the first reduction decreasing
the smallest coefficients until it 
equalizes with the smallest coefficient of $\mathbf{x}$. 
Now we can use the first reduction (and, if necessary, at the very last step the
second reduction) to  equalize the next smallest coefficient with the 
next smallest coefficient of $\mathbf{x}$. Proceeding in the similar manner
will give us $\mathbf{x}$. Here is an example:
to go from $(0,0,0,0)$ to $(2,2,1,1)$, we make the following moves where
$\overset{1}{\to}$ means application of the first reduction and
$\overset{2}{\to}$ means application of the second reduction:
\begin{multline*}
(0,0,0,0)\overset{1}{\to}(1,0,0,0)
\overset{1}{\to} (2,0,0,0)
\overset{1}{\to} (3,0,0,0)
\overset{1}{\to} (4,0,0,0)
\overset{1}{\to} (4,0,0,0)
\overset{1}{\to} \\
(5,0,0,0)
\overset{1}{\to} (6,0,0,0)
\overset{1}{\to} (6,1,0,0)
\overset{1}{\to} (6,2,0,0)
\overset{1}{\to} (6,3,0,0)
\overset{1}{\to} (6,4,0,0)
\overset{1}{\to} \\(6,5,0,0)
\overset{1}{\to} (6,5,1,0)
\overset{1}{\to} (6,5,2,0)
\overset{1}{\to} (6,5,3,0)
\overset{1}{\to} (6,5,4,0)
\overset{1}{\to} (6,5,4,1)
\overset{1}{\to} \\(6,5,4,2)
\overset{1}{\to} (6,5,4,3)
\overset{1}{\to} (6,5,4,2)
\overset{1}{\to} (6,5,4,1)
\overset{1}{\to} (6,5,3,1)
\overset{1}{\to} (6,5,2,1)
\overset{2}{\to} \\(6,5,1,1)
\overset{1}{\to} (6,4,1,1)
\overset{1}{\to} (6,3,1,1)
\overset{1}{\to} (6,2,1,1)
\overset{1}{\to} (5,2,1,1)
\overset{1}{\to} (4,2,1,1)
\overset{1}{\to} \\(3,2,1,1)
\overset{2}{\to} (2,2,1,1).
\end{multline*}
Translated to the language of $\gimel\cdot\mathtt{v}$, we see 
that we can obtain any element  in $\gimel\cdot\mathtt{v}$ from $\mathtt{v}$
using our two reductions. This implies Step~3 and hence completes the proof of the theorem. 
\end{proof}

If $\lambda=(1,2,\dots,n)$, then, under the assumptions of Theorem~\ref{thm77} and for any
$\mathtt{w}\in\gimel\cdot\mathtt{v}$, the module $M_{\mathtt{v}}$  is
the unique simple Gelfand-Zeitlin module extending $\mathbb{C}_{\mathtt{w}}$ as it
has the same Gelfand-Zeitlin character as the module $U_{\lambda}/U_{\lambda}(\mathbf{m}_{{\mathtt{w}}})$.

\vspace{5mm}

\noindent
N.~E.: Department of Mathematics, University of Minnesota, Minneapolis, MN 55455, 
USA, email: {\tt earlnick\symbol{64}gmail.com}

\noindent
V.~M.: Department of Mathematics, Uppsala University, Box. 480,
SE-75106, Uppsala, SWEDEN, email: {\tt mazor\symbol{64}math.uu.se}

\noindent
E.~V.: Departamento de Matem{\'a}tica, Instituto de Ci{\^e}ncias Exatas,
Universidade Federal de Minas Gerais,
Av. Ant{\^o}nio Carlos, 6627, Caixa Postal: 702, CEP: 31270-901, Belo Horizonte, 
Minas Gerais, BRAZIL,
and Laboratory of Theoretical and
Mathematical Physics, Tomsk State University, Tomsk 634050, RUSSIA,
email: {\tt VishnyakovaE\symbol{64}googlemail.com}

\end{document}